\documentclass[reqno,11pt,twoside]{article}
\usepackage{amsmath,amsfonts,amsfonts,amssymb}
\usepackage{amsbsy}
\usepackage{theorem}
\usepackage{mathrsfs,colortbl}
\usepackage{graphicx}
\usepackage{amscd}

\theoremstyle{break}
\newtheorem{lemma}{Lemma}

\newtheorem{proposition}[lemma]{Proposition}
\newtheorem{theorem}[lemma]{Theorem}

\newtheorem{conjecture}[lemma]{Conjecture}

\theorembodyfont{\upshape}
\newtheorem{remark}[lemma]{Remark}
\newtheorem{definition}[lemma]{Definition}
\newtheorem{example}[lemma]{Example}

\newcommand \FF {{\mathbb F}}

\newcommand \AAA {{\mathbb A}}

\newcommand \ZZ {{\mathbb Z}}

\newcommand \cE {{\mathcal E}}

\newcommand \cV {{\mathcal V}}

\newcommand{\q}{/\!\!/}

\newcommand{\Mat}{\mathrm{Mat}}

\begin{document}

\begin{center}
\huge\textsf{On $c_2$ invariants of 4-regular Feynman graphs}\\
\medskip
\normalsize \textsc{ Dmitry Doryn}\\
dmitry@ibs.re.kr

\end{center}
\begin{abstract}
The obstruction for application of effective techniques like denominator reduction for the computation of the $c_2$ invariant of Feynman graphs in general is the absence of a 3-valent vertex for the initial steps. In this paper such a formula for a 4-valent vertex is derived. The formula helps to compute the $c_2$ invariant of new graphs, for instance, 4-regular graphs with small loop number.  
\end{abstract}
\section{Introduction}

It is known that the evaluation of the Feynman periods in $\phi^4$ theory leads to interesting values, usually combinations of multiple zeta values, \cite{BrdKr}, \cite{Sch1}. There is so far no good understanding of these numbers, and they are hard to compute, even numerically.
The $c_2$ invariant can be regarded as a discrete analogue of the period. It is defined as the coefficient of the $q^2$ of the $q$-expansion of the number of $\FF_q$-rational points of the graph hypersurface - the poles of the differential form that is staying in the integral presentation of the period. This invariant respects many relations between the periods.

Starting with a (Feynman) graph $G$ the Feynman period is given by an integral of a differential form with double poles along a graph hypersurface $\cV(\Psi_G)$ given by the graph polynomial 
\begin{equation}\label{e1}
\Psi_G:=\sum_{T} \prod_{e\notin T} \alpha_e\;\in \ZZ[\alpha_1,\ldots,\alpha_{N_G}],
\end{equation}
where $T$ runs over all spanning trees of $G$, $\alpha_i$ are formal variables associated to edges and $N_G$ is the number of edges. For any prime power $q$ a number \#$\cV(\Psi_G)(\FF_q)$ of $\FF_q$-rational points of $\cV(\Psi_G)$ is divisible by $q^2$ for any $G$ with at least 3 vertices. Then define
\begin{equation}\label{e2}
c_2(G)_q:=\cV(\Psi_G)(\FF_q)/ q^2 \mod q.
\end{equation}    
The $c_2$ invariant depends on $q$, but for simple (like denominator reducible) graphs it is just a constant. It can be also a constant outside primes of bad reduction \cite{D2},\cite{Sch2}, or even have a modular nature \cite{BrSch}. The naive way to compute $c_2(G)_q$ for a prime power $q$ is just brute force counting of all the rational points over $\FF_q$ and then take the coefficient of $q^2$. This can be done only for small $q$ and small number of edges and does not compute the whole $c_2$. The much better idea is to try to eliminate the variables step by step and compute only $c_2$ itself, it can be done for many small and for several infinite series of graphs (like zigzag graphs $ZZ_h$) and this procedure is called denominator reduction. Even if the graph is not denominator reducible, it is possible to eliminate a big part of the variables first, decreasing the degree, and then try to compute the rest by other techniques, for example, analysing the underlying geometry, see K3-example in \cite{BrSch}.   

The construction of the Feynman period in $\phi^4$ involves the following operation: one takes a 4-regular graph $\Gamma$ and deletes one of the vertices together with the 4 incident edges (for getting rid of the symmetries), then for the resulting graph $G$ one defines the Feynman period (in parametric space) using the graph polynomial (\ref{e1}). Latter we write $\widehat{G}=\Gamma$ and call it the completion of $G$. Thus a Feynman graph in $\phi^4$ theory has 4 3-valent vertices while the others are 4-valent. In the denominator reduction the first 3 steps (elimination of 3 variables) are special and each next step is just the generic step of denominator reduction.
The ability to apply denominator reduction crucially depends on the existence of a 3-valent vertex for the initial steps, as well as on some factorizations of appearing polynomials that cannot hold in general but do hold for small graphs.

For a graph $G$ with $h_G:=N_G-|V_G|+1$ loops and $N_G$ edges, define $\delta_G:=2h_G-N_G$. The graphs with $\delta_G=0$ (like for example graphs in $\phi^4$) always have a 3-valent vertex, so the denominator reduction can be applied for small graphs. For denominator reducible graphs the reduction gives us the chain of congruence $\mod q$ going down to trivial case, so that $c_2(G) = -1$ for these graphs. We also know that the graphs with $\delta_G<0$ always have $c_2(G)=0$.
 
In the literature there is no one example of graphs with $\delta_G\geq 1$ such that the $c_2$ invariant is known. If $G$ has  $\delta_G\geq 2$, then $G$ could have no 3-valent vertex, thus that denominator reduction technique cannot be applied. This happens for 4-regular graphs. Even with a 3-valent formula,  the analogue of generic step of denominator reduction for a graph with $\delta_G>0$ will give up to 4 terms (not 1) on each step, if the expected factorization occur. 

In this article we derive a formula for the $c_2$ invariant in the case of a 4-valent formula (see Theorem \ref{Theorem3}) as a result of an organized elimination of 4 variables. Based on this we can apply the reduction similar to denominator reduction, that we call semilinear reduction, and compute the $c_2(G)_q$ for some small cases. We find $c_2$ for the completions $\widehat{ZZ}_3,\ldots,\widehat{ZZ}_8$ of graphs of the zigzag series. The other 4-regular graphs are already non-reducible in the sense that not all the appearing polynomials are factorizable, even for first graph $\widehat{P_{6,2}}$ (in notation from \cite{Sch1}).  This is not surprising since $\widehat{G}$ has $h_{\widehat{G}}=h_G+3$ and the graphs in $\phi^4$ also stop to be denominator reducible at 8-9 loops.  Nevertheless, our 4-valent formula and semilinear reduction can be used for elimination of big part of the variables in order to find some nice and understandable geometry behind, like it was done in \cite{BrSch}.

There is an other deeper reason to study the $c_2$ invariant of 4-regular graphs, namely the relation to the completion conjecture (see \cite{BrSch}).
\begin{conjecture}
Let $G_1$ and $G_2$ be two Feynman graphs in $\phi^4$ such that  $\widehat{G_1}=\widehat{G_2}$ (this means they come from the same 4-regular graphs by deletion of two different vertices). Then  $c_2(G_1)=c_2(G_2)$. 
\end{conjecture}

This conjecture is the most interesting statement in the theory of the $c_2$ invariant. It remains unproved.  One of the ideas how to prove it was the following. The fact that $c_2$ invariants are the same for graphs with same completions could have something to do with the completion itself. Since $c_2$ is defined for the completion --- 4-regular graph --- then there could possibly be a way to compare $c_2(G)$ and $c_2(\widehat{G})$. The statement will then follow from non-symmetry of this relation on the vertex we remove.  

In this article this approach is worked out. The formula we get for $c_2$ is symmetric and does not help for proving the conjecture. It produces also no other similar relations for the sub-quotient graphs of 4-regular graphs. 

Finally, for the $c_2$ invariants for 4-regular $\widehat{ZZ_{h}}$ we obtain (Proposition \ref{prop26})
\begin{equation}
c_2(G)_q\equiv -h(h+2) \mod q
\end{equation}
for $h\leq 8$ and we conjecture that this holds for all $h$ (see Conjecture \ref{conj25}).
\medskip

\noindent \textbf{Acknowledgments}\\
The author is thankful to Center for Geometry and Physics (Institute for Basic Science) for the nice working atmosphere. 
The work was supported by the research grant IBS-R003-S1.

\section{Preliminary results}

Let $G$ be a graph with the set of vertices $V=V(G)$ and the set of edges $E=E(G)$. Let $N=N_G:=E(G)$ be the number of edges and let $h_G:=N_G-|V|+1$ be the \textit{loop number}. We numerate edges $e_1,\ldots, e_N$ and associate to each edge $e_i$ the variable (Schwinger parameter) $\alpha_i$. For a given graph, we can define the following polynomial 
\begin{equation}
\Psi_G:=\sum_{T} \prod_{e\notin T} \alpha_e\;\in \ZZ[\alpha_1,\ldots,\alpha_{N_G}],
\end{equation}
where $T$ runs over all spanning trees of $G$, the subgraphs that are trees and contain all the vertices. This polynomial is called the \textit{graph polynomial} (or the first Symanzik polynomial) of $G$. 

The graph polynomial is homogeneous of degree $h_G$ and is linear with respect to each variable. Let $e_k$ be one of the edges. Then $\Psi_G$ can be written in the following way called \textit{contraction-deletion formula}:
\begin{equation}\label{e5}
\Psi_G=\Psi^k_G\alpha_1 +\Psi_{G,k}
\end{equation}
with $\Psi^k_G$ and $\Psi_{G,k}$ independent of $\alpha_k$. It turns out that these two polynomials are again graph polynomials, namely $\Psi^k_G=\Psi_{G\backslash k}$ and $\Psi_{G,k}=\Psi_{G\q k}$, where $G\backslash k$  (respectively $G\q k$) is the graph obtained from $G$ by deletion (respectively contraction) of the edge $e_k$. 
Equivalently, $\Psi_G$ can be defined as the determinant of the matrix.  
\begin{equation}\label{e3}
\Psi_G=\det M_G, \quad M_G =
\left(
\begin{array}{c|c}
\Delta(\alpha)& \cE_G\\
\hline
-\cE_G^{T^{\mathstrut}} & 0 
\end{array}
\right) \in \Mat_{N+n,N+n}(\ZZ[\{\alpha_i\}_{i\in I_N}]),
\end{equation}
where $\Delta(\alpha)$ is the diagonal matrix with entries $\alpha_1,\ldots,\alpha_N$, and $\cE_G\in \Mat_{N,n}(\ZZ)$ is the incidence matrix after deletion of the last column, $N=N_G$, $n=n_G$ (see \cite{Br}, Section 2.2). The equivalence of the two definitions of $\Psi_G$ is the content of the Matrix Tree Theorem.

We should enlarge the set of polynomials we work with.

Let $I, J, K\subset E(G)$ be  sets of edges with $|I|=|J|$ and $K\cap (I\cup J)=\emptyset$.  Out of the matrix $M_G$, one defines the \textit{Dodson polynomials} $\Psi^{I,J}_{G,K}$ by 
\begin{equation}
\Psi^{I,J}_{G,K}:= \det M_G(I;J)_K \ ,
\end{equation} where $M_G(I;J)_K$ obtained from $M_G$ by removing rows indexed by $I$ and columns indexed by $J$, and by putting $\alpha_t=0$ for all $t\in K$. Such a polynomial $\Psi^{I,J}_{G,K}$ is of degree $h_G-|I|$ and depends on $N_G-|I|-|K|$ variables. We usually write $\Psi^I_{K}$ for $\Psi^{I,I}_{G,K}$ and this is consistent with (\ref{e5}). 

Similar to graph polynomials, the Dodgson polynomials satisfy many identities, see \cite{Br}. We give here those that will be used in the sequel. 

For any Dodgson polynomial $\Psi^{I,J}_{G,K}$ and any edge $e_a\in E\setminus I\cup J\cup K$, the \emph{contraction-deletion formula} holds: 
\begin{equation}\label{d8}
\Psi^{I,J}_{G,K} = \pm\Psi^{Ia,Jb}_{G,K}\alpha_a \pm \Psi^{I,J}_{G,Ka}
\end{equation}  
where the signs depend on the choices made for the matrix $M_G$. This formula agrees with $(\ref{e5})$ in the case of the graph polynomial ($I=J$).

Let $I,J\subset E(G)$ be subsets of edges with $I\cap J=\emptyset$ and let $e_a,e_b,e_x\notin I\cup J$. Then 
\begin{equation}\label{d9}
\Psi^{Ix,Jx}\Psi^{Ia,Jb}-\Psi^{Ix,Jb}\Psi^{Ia,Jx}= \gamma\Psi^{I,J}\Psi^{Iax,Jbx}
\end{equation} 
where $\gamma\in\{1, -1\}$ can be understood combinatorically. This is called the \textit{first Dodgson identity}.

Consider again  sets of (pairwise non-intersecting) edges $I$, $J$, $U$ with $|I|=|J|$, and assume $|U|=h_G$. Define by $\cE_G(U)$ the matrix obtained form the incidence matrix $\cE_G$ by removing all the rows corresponding to $U$. We know that $\det(\cE_G(U))=0\; \text{or}\;\pm 1$, and non-zero exactly when $U$ forms a spanning tree. One has the following equality for the Dodgson polynomials
\begin{equation}\label{d10}
\Psi^{I,J}_G= \sum_{U\subset G\backslash I\cup J} \prod_{u\notin U} \alpha_u \det(\cE_G(U\cup I))\det(\cE_G(U\cup J)),
\end{equation}
where $U$ ranges over all subgraphs of $G\backslash (I\cup J)$ which have the property that $U\cup I$ and $U\cup J$ are both spanning tries of $G$, see Proposition 8, \cite{BrY}.

\begin{definition}
Let $P=P_1\cup P_2\cup \ldots \cup P_k$ be a set partition of a subset of vertices of $G$. Define 
\begin{equation}
\Phi^P_G:=\sum_F \prod_{e\notin F} \alpha_e ,
\end{equation} 
where the sum runs over spanning forests $F=T_1\cup \ldots  T_k$, and each tree $T_i$ of $F$ contains the vertixes in $P_i$ and no other vertices of $P$, i.e. $V(T_i)\supseteq P_i$ and $V(T_i)\cup P_j=\emptyset$ for $i\neq j$. These polynomials $\Phi^P_G$ are called the spanning forest polynomials, see \cite{BrY} for examples and more explanation.
\end{definition}
There is an interpretation of the Dodgson polynomials in terms of the spanning forest polynomials, cf. Proposition 12, \cite{BrY}.
\begin{lemma}
For two sets of edges $I$ and $J$ with $|I|=|J|$ and $I\cap J=\emptyset$, one has 
\begin{equation}\label{d12}
\Psi^{I,J}_G=\sum_k \gamma_k\Phi^{P_k}_G.
\end{equation}
Here the sum runs over all partitions of $V(I\cup J)$ and $\gamma_k$ are the coefficients in $\{-1,0,1\}$ that can be also controlled (see Section 2, loc.cit.).
\end{lemma}
We need the \emph{Jacobi's determinant formula}
\begin{lemma}\label{lem3}
Let $M=(a_{ij})$ be an invertible $n\times n$ matrix and let $adj(M)=(A_{i,j})$ be the adjugate matrix of $M$ (the transpose of the matrix of cofactors). Then for any $k$, $1\leq k\leq n$, 
\begin{equation}\label{e27}
\det (A_{ij})_{k\leq i,j\leq n} = \det(M)^{n-k-1}\det (a_{ij})_{1\leq i,j\leq k} \ .
\end{equation}
\end{lemma}
The easy consequence is the following (see \cite{Br}, Lemma 31)
\begin{lemma}
Consider  two sets of edges $I$  and $J$ with $|I|=|J|$ and let $S=\{s_1,\ldots,s_r\}\subset E_G\setminus I\cup J$ be some other subset of edges. If $\Psi^{IS,JS}$ vanishes as a polynomial of $\alpha_i$s, then 
\begin{equation}\label{e28}
\Psi^{Is_t,Js_t}=\sum_{k\neq t}\pm \Psi^{Is_t,Js_k}
\end{equation}
with the signs depending on the order of rows in $M_G$. 
\end{lemma}
The main theorem of the article considers the situation of a 4-valent vertex in a graph and a good way to eliminate first variables.  Here we should discuss the case of a 3-valent vertex that is "classical" but the technique is important for what comes latter. 
\begin{example} \label{main_ex}
Let $G$ be a graph with a 3-valent vertex incident to the edges $e_1$,$e_2$ and $e_3$. Deletion of these 3 edges leads to disconnectedness of the vertex, thus $\Psi^{123}=0$. So we are in the settings of the previous lemma. Other equality to mention is $\Psi^{12}_3=\Psi^{23}_1=\Psi^{13}_2$. This holds since the deletion of any 2 of the 3 edges and contraction of the third one leads to the same subquotient graph. The Jacobi determinant formula  (\ref{e27}) implies 
\begin{equation}
\det \left( \begin{array}{ccc}
\Psi^1 & \Psi^{1,2} & \Psi^{1,3} \\
\Psi^{2,1} & \Psi^2 & \Psi^{2,3} \\
\Psi^{3,1} & \Psi^{3,2} & \Psi^3 \end{array} \right) = 0.
\end{equation}
In the spirit of (\ref{e28}), the first row of the matrix gives $\Psi^1=\Psi^{1,2}-\Psi^{1,3}$. 
We rewrite this in terms of the 3 variables :
\begin{equation}
\Psi^{12}_3\alpha_2+\Psi^{13}_2\alpha_3+\Psi^{1}_{23}= \Psi^{13,23}\alpha_3+\Psi^{1,2}_3+\Psi^{12,23}\alpha_2-\Psi^{1,3}_2.
\end{equation}
Define 
\begin{equation}
f_0:=\Psi^{ij}_{k},\;\: f_k:=\Psi^{i,j}_k,\;\: f_{123}:=\Psi_{123}\;\;\text{ for } \{i,j,k\}=\{1,2,3\} \ .
\end{equation}
Equation above (and similar after permutation of the 3 edges) implies 
\begin{equation}
f_0=\Psi^{ij,k}=\Psi^{ik,jk}\quad \text{and}\quad \Psi^i_{jk} = f_j+f_k 
\end{equation}
for all $\{i,j,k\}=\{1,2,3\}$.
If follows that the graph polynomial for a graph with a 3-valent vertex has the form 
\begin{multline}\label{e30}
\Psi_G=f_0(\alpha_1\alpha_2+\alpha_2\alpha_3+
\alpha_1\alpha_3)+(f_1+f_2)\alpha_3+(f_1+f_3)\alpha_2\\
\quad+ (f_2+f_3)\alpha_1+f_{123},\quad
\end{multline}
together with the identity
\begin{equation}\label{e31}
f_0f_{123}=f_1f_2+f_2f_3+f_1f_3.
\end{equation}
\end{example}

\section{Point-counting functions}
For a prime power $q$ and for an affine variety $Y$ defined over $\ZZ$, we define by $[Y]_q:=\# \bar{Y}(\FF_q)$ the number of $\FF_q$-rational points of $Y$ after extension of scalars to $\FF_q$. More o less, this means that $[f]_q$ is a number of solutions of $f=0$ in $\FF_q^n$ after taking the coefficients$\mod q$ for an affine hypersurface given by $f\in \ZZ[x_1,\ldots,x_n]$. Here and later, we use the shortcut $[f,\ldots,f_n]_q$ for $[\cV(f_1,\ldots,f_n)]_q$. We think of $[\cdot]_q$ as a function of $q$. Sometimes (but not in general), this function is a polynomial of $q$. 

The very basic situation one always meet while computing of the the point-counting function is the case of a linear polynomial or system of two linear polynomials.
\begin{lemma}\label{lemA17}
Let $f^1,f_1,g^1,g_1,h\in\ZZ[\alpha_2,\ldots,\alpha_n]$ be polynomials. Then, considering the varieties on the right hand side of the coming formulas to be in $\AAA^{n-1}$ and the varieties on the left to be in $\AAA^n$, for every $q$,
\begin{enumerate}
\item for $f=f^1\alpha_1+f_1$, one has 
\begin{equation}\label{lin11}
[f,h]_q=[h]_q - [f^1,h]_q + q[f^1,f_1,h]_q \ ,
\end{equation}
and, in particular, 
\begin{equation}\label{lin1}
[f]_q=q^{n-1} - [f^1]_q + q[f^1,f_1]_q \ .
\end{equation}
\item for $f=f^1\alpha_1+f_1$ and $g=g^1\alpha_1+g_1$, one has
\begin{equation}\label{lin21}
[f,g,h]_q=q[f^1,f_1,g^1,g_1,h]_q + [f^1g_1-g^1f_1,h]_q - [f^1,g^1,h]_q
\end{equation}
and
\begin{equation}\label{lin2}
[f,g]_q=q [f^1,f_1,g^1,g_1]_q+ [f^1g_1-g^1f_1]_q - [f^1,g^1]_q \ .
\end{equation}
\end{enumerate}
\end{lemma}
\begin{proof}
See, for example,  Lemma 3.1 in \cite{D4}. 
\end{proof}
\begin{definition}
For 2 polynomials linear in one of the variables, $f=f^k\alpha_k+f_k$ and $g=g^k\alpha_k+g_k$ we use the following notation for there resultant 
\begin{equation}
[f,g]_k:=\pm (f^kg_k-f_kg^k).
\end{equation}
\end{definition}
\begin{lemma}\label{lem10}
Let $G$ be a graph with at least 3 vertices and $q$ is a fixed prime power. Then for the number of rational points on the graph hypersurvace $\cV(\Psi)$ over $F_q$ the following holds:  
\begin{equation}
[\Psi_G]_q\equiv 0 \mod q^2
\end{equation}
and 
\begin{equation}
[\Psi_G^1,\Psi_{G,1}]_q\equiv 0 \mod q
\end{equation}
for any edge $e_1$ and any $q$.
\begin{proof}
See, for example, Theorem 2.9 in \cite{Sch2} and Proposition-Definition 18 in \cite{BrSch}.
\end{proof}
\end{lemma}
This allows us to introduce the main object of our study.
\begin{definition}
Let $G$ be a graph with at least 3 vertices. Then 
\begin{equation}
c_2(G)_q:=[\Psi_G]_q/ q^2 \mod q.
\end{equation}
\end{definition}
We will intensively use the following vanishing argument proved by Katz.
\begin{theorem}[Chevalley-Warning vanishing]\label{CW_thm}
Let $f_1,\ldots, f_k\subset \ZZ[x_1,\ldots,x_n]$ be polynomials and assume that the degrees $d_i:=\deg f_i$ satisfy $\sum_1^k d_i < n$. Then, for the number of $\FF_q$-rational points of the variety $\cV(f_1,\ldots, f_n)$ given by the intersection of the hyperplanes $\cV(f_i)$ in $\AAA^n$, the following congruence holds
\begin{equation}
  [f_1,\ldots,f_k]_q\equiv 0 \mod q.
\end{equation}  
\end{theorem}

For a polynomial $f\in \ZZ[x_1,\ldots,x_N]$ of degree $d$, define $\delta(f):=2d-N$. In the case of a graph polynomial $f=\Psi_G$ the equality $\delta(f)=0$ corresponds to $G$ being log-divergent. If $G$ is a 4-regular graph, then $\delta(\Psi_G)=2$. The positivity of $\delta$ is the obstruction to the vanishing of the big part of the summands  in the reduction procedures of iterative elimination of the variables like denominator reduction.   What we will try to do instead is to keep track of all the summands.\medskip

From the graph polynomial we can always eliminate the first 2 variables.
\begin{lemma}\label{lem13}
Let $G$ be graph with the 2 edges $e_1$ and $e_2$. Then
\begin{equation}\label{e33}
[\Psi_G]_q=q^{N_G-1}-[\Psi^1]_q+ q^2 [\Psi^{12},\Psi^1_2,\Psi^2_1,\Psi_{12}]_q+ q [\Psi^{1,2}]_q - q [\Psi^{12},\Psi^2_1]_q.
\end{equation}
\end{lemma}
\begin{proof}
We use (\ref{lin1}) for $f:=\Psi$ and (\ref{lin2}) for the pair $(\Psi^1,\Psi_1)$ 
\begin{multline}
[\Psi]_q=q^{N_G-1}-[\Psi^{1}]_q+q [\Psi^{1},\Psi_1]_q= q^{N_G-1}-[\Psi^{1}]_q +\\ q^2 [\Psi^{12},\Psi^1_2,\Psi^2_1,\Psi_{12}]_q+ q [\Psi^{12}\Psi_{12}-\Psi^1_2\Psi^2_1]_q - q [\Psi^{12},\Psi^2_1]_q.
\end{multline}
and then the first Dodgson identity (\ref{d9})
\begin{equation}
	\Psi^1_2\Psi^2_1-\Psi^{12}\Psi_{12} = (\Psi^{1,2})^2.
\end{equation}
The statement follows.
\end{proof}

\begin{lemma}\label{lemma11}
Let $G$ be a graph with $|V|\geq 4$ and with a 3-valent vertex, say, incident to the edges $e_1$, $e_2$ and $e_3$.  Then 
\begin{equation}
c_2(G)\equiv [f_0,f_3]_q \mod q.
\end{equation}
\end{lemma}
\begin{proof}
We are going to use (\ref{e33}). Since $G$ has a 3-valent vertex, the graph polynomial $\Psi_G$ has the form (\ref{e30}). In these terms one computes
\begin{equation}
[\Psi^{12},\Psi^1_2,\Psi^2_1,\Psi_{12}]_q=[f_0,f_1+f_3,f_2+f_3, (f_1+f_2)\alpha_3+f_{123}]_q.
\end{equation}
Using (\ref{e31}), one derives that the vanishing $f_1+f_3=0$ on $\cV(f_0)$ implies  $f_1=f_3=0$. So the term above vanishes  mod $q$, since the defining polynomials become  independent of $\alpha_3$.

We also have 
\begin{equation}\label{e35}
[\Psi^{1,2}]_q=[f_0\alpha_3+f_3]_q=q^{N-3} - [f_0]_q + q [f_0,f_3]_q
\end{equation}
and since $f_0$ is again a graph polynomial for $G:=G\backslash 12 \q  3$, it implies $q^2| [f_0]_q$ and
\begin{equation}\label{e35.1}
[\Psi^{1,2}]_q\equiv q [f_0,f_3]_q \mod q^2.
\end{equation}
In the case of a log-divergent $G$, $c_2(G\backslash 12\q 3)=0$.
From the equation 
\begin{equation}\label{e35.5}
[\Psi^2]_q=q^{N-2}-[\Psi^{12}]_q+q [\Psi^{12},\Psi^2_1]_q
\end{equation}
and the fact that the existence of a vertex of valency $\leq 2$ implies the vanishing of the $c_2$ invariant, we derive $q^2| [\Psi^{12},\Psi^2_1]_q$ (and additionally also $q^3| [\Psi^1]_q$). Now  (\ref{e33}) yields
\begin{equation}\label{e35.6}
[\Psi_G]_q\equiv [f_0,f_3]_q q^2 \mod q^3.
\end{equation}
In other words,  for any graph $G$ with $|V|\geq 3$ and a 3-valent vertex,
\begin{equation}\label{e36}
c_2(G)\equiv [f_0,f_3]_q \mod q.
\end{equation}
\end{proof}

There is a way to effectively compute the $c_2$ invariant of a logarithmically divergent graph (i.e. $N_G=2h_G$) that is called denominator reduction. 
 
For a log-divergent graph $G$ (with $e_1$, $e_2$ and $e_3$ incident to a 3-valent vertex), the \emph{denominator reduction} algorithm with respect to some ordering of the edges $e_1,\ldots, e_n$ is a sequence  of polynomials $D_k\in\ZZ[\alpha_{k+1},\ldots,\alpha_n]$ for some $n\leq N-1$ and for all $3\leq k\leq n$ such that  $D_3:=f_0f_3$ and the following rules: 
\begin{enumerate}
\item if $D_{m-1}$ is defined and is factorizable into factors linear in $\alpha_k$
\begin{equation}
      D_{k-1}=(f^k\alpha_k+f_k)(g^k\alpha_k+g_k)
\end{equation}
then the next $D_{k}$ is defined by 
\begin{equation}
	D_k:=\pm\big( f^kg_k-f_kg^k\big)
\end{equation}
\item if $D_{k-1}$ is zero and then $D_k=0$, the algorithm stops and say that $G$ has a weight drop.
\item if $D_k$ is defined for all $k$,  $3\leq k\leq N_G-1$, then $G$ is called denominator reducible.
\end{enumerate}
One can start with $D_0:=\Psi_G$, but the first 3 (initial) steps are special and in this sense we have 2 \emph{initial} weight drops at the first and third step of denominator reduction.  Generically the graph fails to be denominator reducible already at step five: $D_5$ always exists but is not factorizable usually. Nevertheless, all small graphs and many several infinite series of graphs are denominator reducible, see \cite{BrY}. 

\section{$G$ with a 4-valent vertex}

Consider now the situation of a graph $G$ with a 4-valent vertex incident to the edges, say, $e_1,\ldots, e_4$. We are going to think of $\Psi_G$ as polynomial of the 4 corresponding variables $\alpha_1,\ldots, \alpha_4$. We can easily eliminate the first 2 variables using Lemma \ref{lem13} and get 
\begin{equation}
[\Psi_G]_q=q^{N_G-1}-[\Psi^1]_q+ q^2 [\Psi^{12},\Psi^1_2,\Psi^2_1,\Psi_{12}]_q+ q [\Psi^{1,2}]_q - q [\Psi^{12},\Psi^2_1]_q.
\end{equation}
 \bigskip
For further reduction using this formula, we firstly need to understand the summand $[\Psi^{1,2}_G]_q$.
\begin{lemma} \label{lem15}
Let $G$ be as above. Then
\begin{equation}
[\Psi^{1,2}]_q\equiv q\big(-[\Psi^{13,24},\Psi^{14,23}]_q - c_2(G\backslash 3)\\ -c_2(G\backslash 4)\big) \mod q^2.
\end{equation}
\end{lemma}
\begin{proof}
Using contraction-deletion formula (\ref{d8})
\begin{multline}\label{e56}
[\Psi^{1,2}]_q=[\Psi^{13,23}\alpha_3+\Psi^{1,2}_3]_q= q^{N-3}-[\Psi^{13,23}]_q+q[\Psi^{13,23},\Psi^{1,2}_3]_q\\ =q^{N-3}-[\Psi^{13,23}]_q+q^2[\Psi^{134,234}, \Psi^{13,23}_{4}, \Psi^{14,24}_{3},\Psi^{1,2}_{34}]_q\\ +q[\Psi^{134,234}\Psi^{1,2}_{34}-\Psi^{13,23}_{4}\Psi^{14,24}_{3}]_q -q[\Psi^{134,234},\Psi^{14,24}_{3}]_q.
\end{multline}
The most relevant polynomial above is 
\begin{equation}
\Psi^{134,234}\Psi^{1,2}_{34}-\Psi^{13,23}_{4}\Psi^{14,24}_{3}=\Psi^{13,24}\Psi^{14,23},
\end{equation}
here we use the Dodgson identity (\ref{d9}).
The summand $[\Psi^{134,234},\Psi^{14,24}_{3}]_q$ has the form $[f_0,f_3]_q$ for the graph $G\backslash 4$ with a 3-valent vertex, so it supports the $c_2$ invariant.  Similarly to (\ref{e35.1}), $[\Psi^{13,23}]_q$ gives the summand $c_2(G\backslash 3) q$. Putting everything together, one gets 
\begin{multline}
[\Psi^{1,2}]_q\equiv q\big([\Psi^{13,24}]_q+[\Psi^{14,23}]_q-[\Psi^{13,24},\Psi^{14,23}]_q\\ - c_2(G\backslash 3) -c_2(G\backslash 4)\big) \mod q^2.
\end{multline}
The polynomials $\Psi^{13,24}$ and $\Psi^{14,23}$ are both of degree $h_G-2$ and depend on $N_G-4$ varialbes. Since $G$ has at least 4 vertices, $|V_G|-1=N_G-h_G\geq 3$, Chevalley-Warning implies $[\Psi^{13,24}]_q\cong [\Psi^{14,23}]_q\cong 0 \mod q$.
\end{proof}

\begin{proposition}
Let $G$ be a graph with a 4-valent vertex with the 4 incident edges denoted by $e_1,\ldots, e_4$. Then for any prime power $q$,
\begin{equation}\label{e60}
[\Psi_G]_q\equiv q^2\Big([\Psi^{12},\Psi^1_2,\Psi^2_1,\Psi_{12}]_q- \sum _{i=1}^4 c_2(G\backslash i) - [\Psi^{13,24},\Psi^{14,23}]_q\Big) \mod q^3.
\end{equation}
\end{proposition}
\begin{proof}
By Lemma \ref{lem13}, the elimination of the first 2 variables gives
\begin{equation}\label{e53}
[\Psi_G]_q=q^{N_G-1}-[\Psi^1]_q+ q^2 [\Psi^{12},\Psi^1_2,\Psi^2_1,\Psi_{12}]_q+q [\Psi^{1,2}]_q- q[\Psi^{12},\Psi^2_1]_q.
\end{equation}
Both the summands $[\Psi^1]$ and $[\Psi^{12},\Psi^2_1]$ give us $c_2$ invariants of the corresponding graphs with 3-valent vertices (use that $G\backslash 12$ has a 2-valent vertex in (\ref{e35.5})). Together with Lemma \ref{lem15}, this implies the statement.
\end{proof}

The first summand on the right hand side of (\ref{e60}), $[\Psi^{12},\Psi^1_2,\Psi^2_1,\Psi_{12}]_q$, is given by the point-counting function of the intersection of 4 hypersurfaces and can a priori be very complicated. Nevertheless, we can prove that it vanishes$\mod q$.

\begin{proposition}\label{Prop_big} Let $G$ be a graph with a 4 valent vertex with incident edges $e_1,\ldots,e_4$ (and with at least 5 vertices). Then, in the notation above, for any $q$ :
\begin{equation}\label{e68}
[\Psi^{12},\Psi^1_2,\Psi^2_1,\Psi_{12}]_q\equiv  0 \mod q.
\end{equation}
\end{proposition}
The prove is moved to the separate section.\medskip

We finally state the main result of the paper, the 4-valent formula for computation of the $c_2$ invariant.
\begin{theorem}\label{Theorem3}
Let $G$ be a graph with $|V_G|\geq 5$  with a 4-valent vertex with incident edges $e_1\ldots,e_4$. Then, for any $q$ :
\begin{equation}\label{b65}
[\Psi_G]_q\equiv -q^2\Big( [\Psi^{13,24},\Psi^{14,23}]_q  + \sum _{i=1}^4 c_2(G\backslash i)\Big) \mod q^3.
\end{equation}
\end{theorem}
\begin{proof}
Since $[\Psi^{12},\Psi^1_2,\Psi^2_1,\Psi_{12}]_q\equiv  0 \mod q$ by the previous theorem, the statement is the immediate consequence of Formula (\ref{e60}).
\end{proof}

\section{Proof of Proposition \ref{Prop_big}}
The idea and the technique of the proof is very similar to Theorem 4.1, \cite{D4}.  We rewrite $[\Psi^{12},\Psi^1_2,\Psi^2_1,\Psi_{12}]_q$ as a sum of functions that of intersections of at most 2 hypersurfaces, as it is usual for the denominator reduction technique, and then we control all the cancellations.  To do this we look closer to $\Psi_G$, use special notation for some of the appearing graph polynomials and other Dodgson polynomials and we will intensively use the Dodgson identities.

It this section we think of $q$ to be fixed and we \textbf{omit} the index $q$ in the point-counting function and will write \textbf{[$\Psi_G$]} and \textbf{[$f,g$]} instead of $[\Psi_G]_q$ and $[f,g]_q$, this will make our formulas more readable.
 
Since deletion of the all of the first 4 edges disconnects $G$, one has $\Psi^{1234}=0$. Similarly to Example \ref{main_ex}, the Jacobi identity gives us the vanishing of the corresponding $4\times 4$ matrix. The first row implies 
\begin{equation}\label{e65}
\Psi^1=\Psi^{1,2}-\Psi^{1,3}+\Psi^{1,4}.
\end{equation}
Expanding the polynomials in $\alpha_2$, $\alpha_3$ and $\alpha_4$, we obtain  $\Psi^{123}_4=\Psi^{123,234}$, $\Psi^{12}_{34}=\Psi^{12,23}_4-\Psi^{12,24}_3$, $\Psi^1_{234}=\Psi^{1,2}_{34}-\Psi^{1,3}_{24}+\Psi^{1,4}_{23}$.
Let's define 
\begin{equation}
a:=\Psi^{ijk}_t, \quad c^{i,j}=(-1)^{i-j-1}\Psi^{i,j}_{kt},\quad b^i_j:=(-1)^{r_b}\Psi^{ki,it}_j,
\end{equation} 
where $\{i,j,k,t\}=\{1,2,3,4\}$, and $r_b=(k-t)$ if $(k-i)(t-i)>0$, and $r_b=(k-t-1)$ otherwise.
Managing other rows of the matrix similar to (\ref{e65}), we derive
\begin{equation}\label{e66}
\begin{aligned}
\Psi^{ijk,ijt}&=a=\Psi^{ijk}_t,\\ 
\Psi^{ij}_{kt}&=b^i_k+b^i_t,\\
\Psi^i_{jkt}&=c^{i,j}+c^{i,k}+c^{i,t}
\end{aligned}
\end{equation}
for all $\{i,j,k,t\}=\{1,2,3,4\}$. The Dodgson identities for the polynomials above imply 
\begin{equation}\label{e67}
(b^i_t)^2\equiv \Psi^{ij}_{kt}\Psi^{ik}_{jt} \mod a \;\;\subset \ZZ [ \alpha ].
\end{equation}
We need to prove the following identity.
\begin{lemma}
Let $G$ be a graph as above. In terms of polynomials defined above,
\begin{equation}\label{e8888}
\Psi^{12,34}=b^2_4-b^1_4=b^2_3 - b^1_3.
\end{equation}
\end{lemma}
\begin{proof}
We will have the following picture with labelled vertices and edges in mind, see Figure 1. 
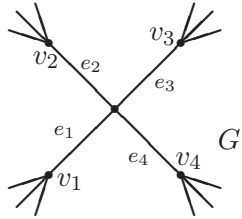
\begin{figure}[h]
\centering
\begin{picture}(100,80)
\thicklines
\put(08,08){\line(1,1){74}}
\put(08,82){\line(1,-1){74}}
\put(20,20){\line(-3,-1){15}}
\put(20,20){\line(-1,-3){5}}
\put(70,70){\line(3,1){15}}
\put(70,70){\line(1,3){5}}
\put(20,70){\line(-1,3){5}}
\put(20,70){\line(-3,1){15}}
\put(70,20){\line(1,-3){5}}
\put(70,20){\line(3,-1){15}}
\put(20,20){\circle*{3}}
\put(70,20){\circle*{3}}
\put(20,70){\circle*{3}}
\put(70,70){\circle*{3}}
\put(45,45){\circle*{3}}
\put(14,62){\text{$v_2$}}\put(58,71){\text{$v_3$}} \put(68,24){\text{$v_4$}}\put(23,16){\text{$v_1$}}

\put(32,60){\text{\scriptsize{$e_2$}}}\put(60,53){\text{\scriptsize{$e_3$}}} \put(50,25){\text{\scriptsize{$e_4$}}}\put(22,35){\text{\scriptsize{$e_1$}}}
\put(84,30){\text{$G$}}
\thinlines 
\end{picture}
\caption{$G$ with a 4-valent vertex}
\end{figure}

Using (\ref{d10}), one writes 
\begin{equation}\label{d40}
\Psi^{12,34}_G= \sum_{U\subset G\backslash I\cup J} \prod_{u\notin U} \alpha_u \det(\cE_G(U\cup \{12\}))\det(\cE_G(U\cup \{34\})),
\end{equation}
where $U$ ranges over all subgraphs of $G\backslash 1234$ which have the property that $U\cup \{12\}$ and $U\cup \{34\}$ are both spanning tries of $G$. One easily sees that the subgraph $U$ should necessary consist of two trees and the vertices $v_1$ and $v_2$ (resp. $v_3$ and $v_4$) are on the different trees. Hence, we have only 2 possibilities and in terms of (\ref{d12}) one obtains
\begin{equation}\label{d43}
\Psi^{12,34}_G=\Phi^{\{13\},\{24\}}_G - \Phi^{\{23\},\{14\}}_G \ .
\end{equation}
A reader can think that $\Phi^{\{13\},\{24\}}_G$ vanishes by the intuition coming from Figure 1, but this is true only form planar graphs.
Analysing similarly $\Psi^{12,13}_4$ and $\Psi^{12,23}_4$, one finds
\begin{equation}\label{d44}
\begin{aligned}
b^1_4=\Phi^{\{23\},\{1,4\}}_{G}+\Phi^{\{123\},\{4\}}_{G}\\
b^2_4=\Phi^{\{13\},\{1,4\}}_{G}+\Phi^{\{123\},\{4\}}_{G}
\end{aligned}
\end{equation}
The first statement follows from (\ref{d43}) and (\ref{d44}). The second statement trivially follows from this using the middle equation of (\ref{e66}).
\end{proof}
The statement holds for all permutation of indexes also modulo the correct signs, but here we only need that form stated above.


We are going to use a formula for elimination of one of the variables from the system of many polynomials linear in that variable, see Proposition 29, \cite{BSY}. Let $f_1,\ldots,f_k\in \ZZ[\alpha_1,\ldots,\alpha_n]$ with $f_i=f_i^1\alpha+f_{i,1}$ for $\alpha=\alpha_1$. Then 
\begin{multline}\label{e70}
[f_1,\ldots,f_n]=[f^\alpha_1,f_{1,\alpha},\ldots,f^\alpha_n,f_{n,\alpha}]q +\\ [[f_1,f_2]_\alpha,\ldots,[f_1,f_n]_\alpha] - [f^\alpha_1,\ldots,f^\alpha_n]\\
\sum^{n-2}_{k=1}([f^\alpha_1,f_{1,\alpha}\ldots,f^\alpha_k,f_{k,\alpha},[f_{k+1},f_{k+2}]_{\alpha},\ldots,[f_{k+1},f_n]_{\alpha}]\\
-[f^\alpha_1,f_{1,\alpha}\ldots,f^\alpha_k,f_{k,\alpha}]).
\end{multline}
Here $[f,g]_{\alpha}=\pm (f^ig_i-f_ig^i)$  is the resultant with respect to $\alpha_i$ for polynomials $f$ and $g$ linear in  $\alpha$. We also write $[f,g]_i$ for $[f,g]_{\alpha_i}$ sometimes.
We apply Formula (\ref{e70}) to the right hand side of (\ref{e68}) for the variable $\alpha=\alpha_3$:
\begin{equation}
f_a=\Psi^{12},\; f_b=\Psi^1_2,\; f_c=\Psi^2_1\;, f_d=\Psi_{12}.
\end{equation}
One obtains
\begin{multline}\label{e73}
[\Psi^{12},\Psi^1_2,\Psi^2_1,\Psi_{12}]=
[f_a,f_b,f_c,f_d]=[f^3_a,f_{a3},f^3_b,f_{b3},f^3_c,f_{c3},f^3_d,f_{d3}]q\\ + \big(A+B+C\big) - \big([f^3_a,f^3_b,f^3_c,f^3_d]+[f^3_a,f_{a3}]+[f^3_a,f_{a3},f^3_b,f_{b3}]\big),
\end{multline}
where
\begin{equation}\label{e74}
\begin{aligned}
A&=\big[[f_a,f_b]_3,[f_a,f_c]_3,[f_a,f_d]_3\big],\\
B&=\big[f^3_a,f_{a3},[f_b,f_c]_3,[f_b,f_d]_3\big],\\
C&=\big[f^3_a,f_{a3},f^3_b,f_{b3},[f_c,f_d]_3\big].
\end{aligned}
\end{equation}
Each of the three summands in the last brackets of (\ref{e73}) is divisible by $q$. Indeed, the variety $\cV(f^3_a,f^3_b,f^3_c,f^3_d)\subset\AAA^{N-2}$ is the cone over the variety defined by the same equations  but in $\AAA^{N-3} (\text{no}\; \alpha_3)$, thus $q|[f^3_a,f^3_b,f^3_c,f^3_d]$. Now $[f_a^3,f_{a3}]=[\Psi^{123},\Psi^{12}_3]=[\Psi^3_{G'},\Psi_{G',3}]$ for $G'=G\q 12$, so $q |[f_a^3,f_{a3}]$ by Lemma \ref{lem10}. For the last summand $[f^3_a,f_{a3},f^3_b,f_{b3}]=[\Psi^{123},\Psi^{12}_3,\Psi^{13}_2,\Psi^1_{23}]$ we are going to use the formula for a 3-valent vertex (\ref{e30})  for the graph $G':=G\q 1$ with edges $e_2,e_3,e_4$ forming the 3-valent vertex. In the notation with $f_i$s but with indices $i=2,3,4$, we have 
\begin{multline}\label{e75}
[\Psi^{23}_{G'},\Psi^{2}_{G',3},\Psi^{3}_{G',2},\Psi_{G',23}]=[f_0,f_0\alpha_3+(f_3+f_4),g_0\alpha_3+(f_2+f_3),\\ (f_2+f_4)\alpha_3+f_{234}]=[f_0,f_3+f_4,f_2+f_3,(f_2+f_4)\alpha_3+f_{234}].
\end{multline}
The connecting identity (\ref{e31}) is of the form $f_0f_{234}=f_2(f_3+f_4)+f_3f_4$, thus again the vanishing of $f_3+f_4$ on $\cV(f_0)$ implies the vanishing of both summands $f_3$ and $f_4$. Analogously, 
\begin{equation}\label{b32}
[f_0,f_2+f_3]=[f_0,f_2,f_3].
\end{equation}
It follows now that all the terms in the brackets 
(\ref{e75}) become independent of $\alpha_3$. As a consequence, it gives us a cone over a variety in $\AAA^{N-3}$, thus the number of points is divisible by $q$.

Summarizing, we derive the following congruence from (\ref{e73}):
\begin{equation}\label{d36}
[\Psi^{12},\Psi^1_2,\Psi^2_1,\Psi_{12}]\equiv\big(A+B+C\big) \mod q
\end{equation} 
with $A,B,C$ given by (\ref{e74}). Now we will work with these 3 summands separately and then will show that they sum up to $0 \mod q$. For simplicity, we list here the involved polynomials:
\begin{equation}\label{e79}
\begin{aligned}
\mathstrut[f_a,f_b]_3&=\Psi^{123}\Psi^1_{23}-\Psi^{12}_3\Psi^{13}_2=(\Psi^{12,13})^2=(a\alpha+b^1_4)^2,\\
[f_a,f_c]_3&=\Psi^{123}\Psi^2_{13}-\Psi^{12}_3\Psi^{23}_1=(\Psi^{12,23})^2=(a\alpha+b^2_4)^2,\\
[f_c,f_d]_3&=\Psi^{23}_1\Psi_{123}-\Psi^2_{13}\Psi^3_{12}=(\Psi^{2,3}_1)^2,\\
[f_b,f_d]_3&=\Psi^{13}_2\Psi_{123}-\Psi^3_{12}\Psi^1_{23}=(\Psi^{1,3}_2)^2,\\
[f_b,f_c]_3&=\Psi^{13}_2\Psi^2_{13}-\Psi^{23}_1\Psi^1_{23},\\
[f_a,f_d]_3&=\Psi^{123}\Psi_{123}-\Psi^{12}_3\Psi^3_{12}.
\end{aligned}
\end{equation}
The coefficient of $\alpha_2$ in the expansion of the first Dodgson identity $\Psi^1_3\Psi_1^3-\Psi^{13}\Psi_{13}=(\Psi^{1,3})^2$ in $\alpha_2$ gives 
\begin{equation}\label{e81}
\Psi^{12}_3\Psi^3_{12}+\Psi^1_{23}\Psi^{23}_1- \Psi^{123}\Psi_{123}-\Psi^{13}_2\Psi^2_{13} = -2\Psi^{12,23}\Psi^{1,3}_2.
\end{equation} 
Similarly, for the expansion in $\alpha_1$ of the Dodgson identity for the pair of edges $e_2$ and $e_3$ implies
\begin{equation}\label{e82}
\Psi^{12}_3\Psi^3_{12}+\Psi^2_{13}\Psi^{13}_2- \Psi^{123}\Psi_{123}-\Psi^{23}_1\Psi^1_{23} = 2\Psi^{12,13}\Psi^{2,3}_1.
\end{equation} 
The sum of the two equalities above reads 
\begin{equation}\label{e83}
\Psi^{12}_3\Psi^3_{12}-\Psi^{123}\Psi_{123}= \Psi^{12,13}\Psi^{2,3}_1-\Psi^{12,23}\Psi^{1,3}_2. 
\end{equation} 
It follows that $[f_a,f_d]_3\in \ZZ[\alpha]$ is in the ideal generated by $\Psi^{12,13}$ and $\Psi^{12,23}$. One computes
\begin{multline}\label{e84}
A=[[f_a,f_b]_3,[f_a,f_c]_3,[f_a,f_d]_3]=[\Psi^{12,13},\Psi^{12,23}]=\\ [a\alpha_4+b^1_4,a\alpha_4+b^2_4]=[a\alpha_4+b^1_4,b^2_4-b^1_4].
\end{multline}
Similarly to the Pl\"uker identity,  Lemma 27 in \cite{Br}, one  obtains: 
\begin{equation}
\det M_G(\{1,2\},\{3,4\})-\det M_G(\{1,2\},\{1,3\})+\det M_G(\{1,2\},\{2,3\})=0.
\end{equation}
Expansion in $\alpha_4$ gives
\begin{equation}\label{e86}
\Psi^{12,34}=b^2_4-b^1_4.
\end{equation}
After the elimination of $\alpha_4$ by (\ref{lin11}),  the equalities (\ref{e84}) and (\ref{e86}) imply
\begin{equation}\label{e87}
A\equiv [\Psi^{12,34}]-[a,\Psi^{12,34}] \mod q.
\end{equation}
Now we are going to compute $B$:
\begin{equation}
B=[f^3_a,f_{a3},[f_b,f_c]_3,[f_b,f_d]_3]=[a,\Psi^{12}_{34},[f_b,f_c]_3,\Psi^{1,3}_2]. 
\end{equation}
We use again the equalities (\ref{e81}) and (\ref{e82}) but now subtracting instead of adding. We immediately get
\begin{equation}
[f_b,f_c]_3=[\Psi^{13}_2\Psi^2_{13}-\Psi^{23}_1\Psi^1_{23}]= \Psi^{12,13}\Psi^{2,3}_1+\Psi^{12,23}\Psi^{1,3}_2.
\end{equation}
It follows that 
\begin{multline}
B=[a,\Psi^{12}_{34},\Psi^{1,3}_2,\Psi^{12,13}\Psi^{2,3}_1]=[a,\Psi^{12}_{34},\Psi^{1,3}_2,(a\alpha_4+b^1_4)\Psi^{2,3}_1]\\
=[a,\Psi^{12}_{34},b^4_2\alpha_4+\Psi^{1,3}_{24},b^1_4\Psi^{2,3}_1].
\end{multline} 
The last term of the last brackets disappears, this follows from (\ref{e67}): $b^1_4$ vanishes on $\cV(a,\Psi^{12}_{34})$.
By (\ref{lin11}), eliminating $\alpha_4$, one now computes 
\begin{equation}
B\equiv [a,\Psi^{12}_{34}] - [a,\Psi^{12}_{34},b^4_2]\mod q.
\end{equation}
The third summand of (\ref{d36}), $C$, takes the form  
\begin{multline}\label{e89}
C=[f^3_a,f_{a3},f^3_b,f_{b3},[f_c,f_d]_3]=[a,\Psi^{12}_{34},\Psi^{13}_{24},\Psi^1_{23},\Psi^{2,3}_1]=\\ [a,\Psi^{12}_{34},\Psi^{13}_{24},\Psi^{14}_{23}\alpha_4+\Psi^1_{234},b^4_1\alpha_4+\Psi^{2,3}_{14}].
\end{multline}
We claim that $\Psi^{14}_{23}$ lies in the ideal generated by $a,\Psi^{12}_{34},\Psi^{13}_{24}$. Indeed, $\Psi^{14}_{23}=b^1_2+b^1_3$ and, by (\ref{e67}), $b^1_2$ vanishes on $\cV(a,\Psi^{13}_{24})$ while $b^1_3$ vanishes on $\cV(a,\Psi^{12}_{34})$. Thus only the last polynomial in (\ref{e89}) depends on $\alpha_4$. One computes 
\begin{equation}
C\equiv [a,\Psi^{12}_{34},\Psi^{13}_{24},\Psi^1_{234}]-[a,\Psi^{12}_{34},\Psi^{13}_{24},\Psi^1_{234},b^4_1]\mod q.
\end{equation}  
Consider the equation similar to (\ref{e83}) but for the collection of edges $(e_1,e_2,e_4)$ instead of $(e_3,e_1,e_2)$: 
\begin{equation}
\Psi^{24}_1\Psi^1_{24}-\Psi^{124}\Psi_{124}= \Psi^{14,24}\Psi^{1,2}_4-\Psi^{12,24}\Psi^{1,4}_2.
\end{equation}
Each of the appearing polynomials depends on $\alpha_3$. A consideration of the constant coefficient gives
\begin{equation}\label{e91}
\Psi^{24}_{13}\Psi^1_{234}-a\Psi_{1234}=b^4_3\Psi^{1,2}_{34}-b^2_3\Psi^{1,4}_{23}.
\end{equation}
Consider the variety $Z=\cV(a,\Psi^{12}_{34},\Psi^{13}_{24})\subset \AAA^{N_G-4}$ and let $Y=Z\backslash Z\cap \cV(b^4_1)$. Since the vanishing of $\Psi^{12}_{34}$ implies $b^1_3=0$ and the vanishing of $\Psi^{13}_{24}$ implies $b^1_2=0$ on $\cV(a)$ by (\ref{e67}), one gets also $\Psi^{14}_{23}=b^1_2+b^1_3=0$ on $\cV(a)$. Hence, again by (\ref{e67}), $b^4_3$ vanishes on $Z$. The equation (\ref{e91}) now implies $\Psi^{24}_{13}\Psi^1_{234}=0$ on $Z$. Since $\Psi^{24}_{13}=b^4_1+b^4_3$, and $b^4_3=0$ while $b^4_1\neq 0$ on $Y$, one derives $Y\cap \cV(\Psi^1_{234})\cong Y$. Thus $C=[\cV(a,\Psi^{12}_{34},\Psi^{13}_{24},\Psi^1_{234})\backslash \cV(a,\Psi^{12}_{34},\Psi^{13}_{24},\Psi^1_{234},b^4_1)]=[Y]$. One computes 
\begin{multline}\label{e92}
B+C\equiv ([a,\Psi^{12}_{34}] + [a,\Psi^{12}_{34},\Psi^{13}_{24}])\\ -( [a,\Psi^{12}_{34},b^4_2]+ [a,\Psi^{12}_{34},\Psi^{13}_{24},b^4_1]) \mod q.
\end{multline}
For $[a,\Psi^{12}_{34},b^4_2]$, one uses the equality  $(b^4_2)^2\equiv \Psi^{14}_{23}\Psi^{34}_{12}\mod a$ in (\ref{e67}) and gets
\begin{multline}\label{e93}
[a,\Psi^{12}_{34},b^4_2]=[a,\Psi^{12}_{34},\Psi^{14}_{23}\Psi^{34}_{12}]=[a,\Psi^{12}_{34},\Psi^{14}_{23}]+[a,\Psi^{12}_{34},\Psi^{34}_{12}]\\
-[a,\Psi^{12}_{34},\Psi^{14}_{23},\Psi^{34}_{12}].
\end{multline}
Similarly, 
\begin{multline}\label{e94}
[a,\Psi^{12}_{34},\Psi^{13}_{24},b^4_1]=[a,\Psi^{12}_{34},\Psi^{13}_{24},\Psi^{24}_{13}\Psi^{34}_{12}]=[a,\Psi^{12}_{34},\Psi^{13}_{24},\Psi^{24}_{13}]+\\
[a,\Psi^{12}_{34},\Psi^{13}_{24},\Psi^{34}_{12}]-[a,\Psi^{12}_{34},\Psi^{13}_{24},\Psi^{24}_{13},\Psi^{34}_{12}].
\end{multline}
The last summands of (\ref{e93}) and (\ref{e94}) coincide. Indeed, $\Psi^{12}_{34}=0=\Psi^{13}_{24}$ on $\cV(a)$ imply $\Psi^{23}_{14}=0$ since $e_1,e_2,e_3$ form a 3-valent vertex in $G\q 4$, and also $\Psi^{23}_{14}=0=\Psi^{34}_{12}$ imply $\Psi^{24}_{13}=0$ for the 3-valent vertex formed by $e_2,e_3,e_4$ in $G\q 1$. One derives
\begin{multline}
[a,\Psi^{12}_{34},\Psi^{13}_{24},\Psi^{24}_{13},\Psi^{34}_{12}]=[a,\Psi^{12}_{34},\Psi^{13}_{24},\Psi^{34}_{12}]=[a,\Psi^{12}_{34},\Psi^{13}_{24},\Psi^{24}_{13}]\\
=[a,\Psi^{12}_{34},\Psi^{14}_{23},\Psi^{34}_{12}].
\end{multline}
Hence, 
\begin{multline}\label{e95}
B+C\equiv [a,\Psi^{12}_{34}] + [a,\Psi^{12}_{34},\Psi^{13}_{24}]-[a,\Psi^{12}_{34},\Psi^{14}_{23}]\\-[a,\Psi^{12}_{34},\Psi^{34}_{12}] \mod q.
\end{multline}
The first summand on the right hand side is divisible by $q$ by Lemma \ref{lem10} applied to $[\Psi^{3}_{G'},\Psi_{G',3}]$ for $G'=G\backslash 4\q \{1,2\}$. Similarly, the second summand on the right hand side of the equality 
\begin{equation}\label{e97}
[a,\Psi^{12}_{34},\Psi^{13}_{24}]=[a,\Psi^{12}_{34}]+[a,\Psi^{13}_{24}]-[a,\Psi^{12}_{34}\Psi^{13}_{24}]
\end{equation}
is divisible by $q$. Using the equality (\ref{e67}), one gets 
\begin{equation}
[a,\Psi^{12}_{34},\Psi^{13}_{24}]\equiv -[a,\Psi^{12}_{34}\Psi^{13}_{24}]\equiv -[a,b^1_4] \mod q.
\end{equation}
The same thing can be done with $[a,\Psi^{12}_{34},\Psi^{14}_{23}]$ in (\ref{e95}). One can also do the step (\ref{e97}) for $[a,\Psi^{12}_{34},\Psi^{34}_{12}]$. The congruence (\ref{e95}) now implies
\begin{equation}
B+C\equiv [a,b^1_3]-[a,b^1_4] + [a,\Psi^{12}_{34}\Psi^{34}_{12}] \mod q. 
\end{equation}
By (\ref{d36}) and (\ref{e87}), we finally get the desired formula
\begin{multline}
[\Psi^{12},\Psi^1_2,\Psi^2_1,\Psi_{12}]\equiv A+B+C\equiv [\Psi^{12,34}]-[a,\Psi^{12,34}]\\ +[a,b^1_3]-[a,b^1_4] + [a,\Psi^{12}_{34}\Psi^{34}_{12}] \mod q.
\end{multline}

\bigskip
One computes $[a,b^1_3]=[\Psi^{23,34}_{G'},\Psi^{2,4}_{G',3}]\equiv c_2(G')\mod q$ for $G'=G\backslash 1$, since the additional (see (\ref{e35.6})) summand $c_2(G'\backslash 3)$ is vanishes $\mod q$ by the existence of a 2-valent vertex. The same holds for $[a,b^1_4]$, so these terms sum up to 0 mod $q$. 
One also has $q|[\Psi^{12,34}]$ by Chevalley-Warning since $|V_G|\geq 5$. 
Hence,
\begin{equation}\label{e102}
[\Psi^{12},\Psi^1_2,\Psi^2_1,\Psi_{12}]\equiv [a,\Psi^{12}_{34}\Psi^{34}_{12}]-[a,\Psi^{12,34}] \mod q.
\end{equation}
Let's consider a modification of our graph shown on the Figure 2.
\begin{figure}[h]
\centering
\begin{picture}(270,80)
\thicklines
\put(08,08){\line(1,1){74}}
\put(08,82){\line(1,-1){74}}
\put(20,20){\line(-3,-1){15}}
\put(20,20){\line(-1,-3){5}}
\put(70,70){\line(3,1){15}}
\put(70,70){\line(1,3){5}}
\put(20,70){\line(-1,3){5}}
\put(20,70){\line(-3,1){15}}
\put(70,20){\line(1,-3){5}}
\put(70,20){\line(3,-1){15}}
\put(20,20){\circle*{3}}
\put(70,20){\circle*{3}}
\put(20,70){\circle*{3}}
\put(70,70){\circle*{3}}
\put(45,45){\circle*{3}}

\put(180,20){\line(-3,-1){15}}
\put(180,20){\line(-1,-3){5}}
\put(230,70){\line(3,1){15}}
\put(230,70){\line(1,3){5}}
\put(180,70){\line(-1,3){5}}
\put(180,70){\line(-3,1){15}}
\put(230,20){\line(1,-3){5}}
\put(230,20){\line(3,-1){15}}
\put(180,20){\line(1,0){50}}
\put(180,70){\line(1,0){50}}
\put(180,20){\line(-1,-1){12}}
\put(180,70){\line(-1,1){12}}
\put(230,20){\line(1,-1){12}}
\put(230,70){\line(1,1){12}}
\put(180,20){\circle*{3}}
\put(180,70){\circle*{3}}
\put(230,20){\circle*{3}}
\put(230,70){\circle*{3}}
\put(27,64){\text{$e_1$}}\put(62,55){\text{$e_2$}} \put(52,24){\text{$e_3$}}\put(18,30){\text{$e_4$}}
\put(84,30){\text{$G$}}
\put(110,40){\vector(1,0){30}}
\put(240,30){\text{$G'$}}
\put(200,25){\text{$e_t$}}\put(205,60){\text{$e_s$}}
\thinlines 
\end{picture}
\caption{From $G$ to $G'$.}
\end{figure}
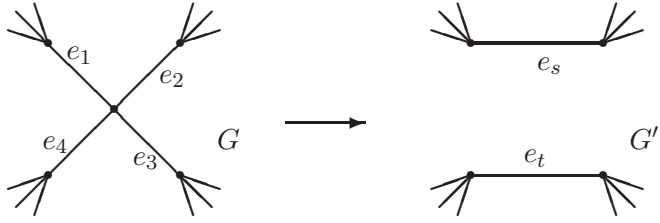
In a graph $G$  with edges $e_1,\ldots,e_4$ forming a 4-valent vertex, we remove this 4 edges (and the vertex) and add two other edges $e_s$ and $e_t$. We denote the resulting graph by $G'$. One easily sees that 
\begin{equation}
\Psi^{12}_{G,34}=\Psi^s_{G',t}\quad \text{and}\quad \Psi^{34}_{G,12}=\Psi^t_{G',s}.
\end{equation}
Using the first Dodgson identity for $I=\{s\}$, $J=\{t\}$, one gets
\begin{equation}\label{b60}
\cV(a,\Psi^{12}_{G,34}\Psi^{34}_{G,12})\cong\cV(\Psi^{st}_{G'},\Psi^s_{G',t}\Psi^t_{G',s})\cong\cV(\Psi^{st}_{G'},\Psi^{s,t}_{G'}).
\end{equation}
Now it remains to show that $\Psi^{s,t}$ and $\Psi^{12,34}$ coincide. 

One can again apply Formula (\ref{d40}) for both  $\Psi^{s,t}$ and $\Psi^{12,34}$. It is easy to see that for each $U$ such that $U\cup \{1,2\}$ (resp. $U\cup \{3,4\}$)  is a spanning tree for $G$, one also has $U\cup \{s\}$ (resp. $U\cup \{t\}$) is a spanning tree for $G'$ and vice-versa. The coefficients in the formula above also coincide, so we get
\begin{equation}
\Psi^{12,34}_G=\Psi^{s,t}_{G'}.
\end{equation}
As a consequence,
\begin{equation}
[a,\Psi^{s,t}_{G'}] = [a,\Psi^{12,34}_{G}] 
\end{equation}
By (\ref{e102}), we finally obtain
\begin{equation}
[\Psi^{12},\Psi^1_2,\Psi^2_1,\Psi_{12}]\equiv 0 \mod q,
\end{equation}
as desired. This finishes the proof of Proposition \ref{Prop_big}.

\section{4-regular graphs}
In this section we look closer to the case of our graph $G$ being 4-regular. In contrary to the situation of a log-divergent graph, we cannot use formula for 3-valent vertex to compute the $c_2$ invariant. The only one possibility that we have now is to use Theorem \ref{Theorem3}. 

\begin{definition}
For a pair of homogeneous polynomials $f,g\in \ZZ[x_1,\ldots,x_n]$ we denote by $\deg_{tot} (f,g)$ and $N(f,g)$ the total degree and the total number of variables they depend on. Also define $\delta(f,g)=\deg_{tot} (f,g)-N(f,g)$. We also use the similar notation for a single homogeneous polynomial $f$ (assume $g=0$). 
\end{definition}

\begin{definition}\label{defSLR}
Define the set of \emph{semilinearly reducible} homogeneous polynomials and (unordered) pairs of homogeneous polynomials $SLR$ inductively:
\begin{itemize}
\setlength\itemsep{0.5em}
\item[$\diamond$] a linear polynomial $f$ with $N(P)=1$ is in $SLR$ and is called elementary. 
\item[$\diamond$] a polynomial  $f$ with $\delta(f)<0$ is in $SLR$ and is called elementary.
\item[$\diamond$] a pair $(f, g)$ of polynomials with $\delta(f,g)<0$ is in $SLR$ and is called elementary.

\item[$\diamond$] a pair $(f, m x_1)$ with $m\in\ZZ$ and $N(f)\geq 1$ is in $SLR$ if  $mf_1$ is in $SLR$.
\item[$\diamond$] a polynomial $f$ with $\delta\geq 0$ is in $SLR$ if there exist a variable $x_1$ 
such that $f$ is linear in $x_1$ and $f^1$ and is in $SLR$. 
\item[$\diamond$] a polynomial $h$ with $\delta\geq 0$ is in $SLR$ if there exist a factorization $h=f^2g$ such that  $\deg_{tot} f\geq 1$  and $fg$ is in $SLR$.
\item[$\diamond$] a pair $(f,g)$ with $\delta\geq 0$ and $N(f),N(g)>1$ is in $SLR$ if there exist a variable $x_1$ such that   both $f$ and $g$ are linear in $x_1$, $(f^1,g^1)\in SLR$ and $f^1g_1-f_1g^1\in SLR$. 
\item[$\diamond$] a polynomial $h$ is in $SLR$ if there exist a factorization $h=fg$ such that  both $f$ and $g$ are in $SLR$, and also $(f,g)\in SLR$.
\end{itemize}
\end{definition}
Since our operation $(f,g)\mapsto (f^1,g^1)$ and $(f,g)\mapsto [f,g]_1$ decrease the total degree by 1, the set $SLR$ is correctly defined by induction. The definition realizes a simple wish to compute the $c_2$ invariant using similar technique to the denominator reduction but for polynomials $\Psi_G$ for $2h_G>N_G$. In order to do that we need is to control all the pieces, not just resultants.  We have an additional structure for each element $D$ in $SLR$, namely the tree $T(D)$ of elements of $SLR$ that give a reduction of $D$ down to the elementary elements (leaves of $T$).
\begin{remark}
The notion of semilinear reducibility is stronger then denominator reducibility, since we should control factorization of more polynomials. On the other side, it is in some sense less strong then the linear reducibility used in the articles discussing evaluation of Feynman integrals by use of iterated integrals. 
\end{remark}
\begin{definition}
A graph $G$ is called \emph{semilinearly reducible} if at least one of the conditions below is satisfied :
\begin{itemize}
\item[$\diamond$] $G$ has a 3-valent vertex (say, incident to edges $e_1,e_2,e_3$) such that the pair of polynomials $(\Psi^{1,2}_{3},\Psi^{13,23})$ is semilinearly reducible, i.e. an element of $SLR$. 
\item[$\diamond$] $G$ has a 4-valent vertex (say, incident to edges $e_1,\ldots,e_4$) such that the pair of polynomials $(\Psi^{13,24},\Psi^{14,23})$, together with 4 pairs $(\Psi^{i,j}_{kt},\Psi^{ik,jk}_t)$ for $(i,j,k,t)$ cyclic permutation of $(1,2,3,4)$ are all semilinearly reducible.
\end{itemize} 
\end{definition}

\begin{proposition}
Let $G$ be a graph with at least one vertex of valency $\leq 4$ and assume $G$ is semilinearly reducible. Then the $c_2$ invariant can be computed in finite number of steps of semilinear reduction by elements in $SLR$:  apart from possible primes of bad reduction
\begin{equation}
\exists\; c\in \ZZ\quad\text{ such that }\quad c_2(G)_q \equiv -c\mod q  \quad \text{for all}\; q.
\end{equation}
Moreover,  $|c|< \frac{1}{2}4^{h_G}$.
\end{proposition}
\begin{proof}
If a graph $G$ has a 2-valent vertex, we know that $c_2(G)=0$.
If a graph $G$ has a 3-valent vertex incident to $e_1,e_2,e_3$, then by Lemma \ref{lemma11} we know that the $c_2$ invariant congruent ot the zeroth coefficient of $[f_0,f_3]$. Similarly, for a 4-valent vertex incident to $e_1,\ldots,e_4$, we use the formula $\ref{b65}$ with 5 summands.

If $2h_G<N_G$, then the term for 3-valent vertex case or each of the 5 terms for 4-valent case is zero$\mod q$ by Chevalley-Warning, since $\deg_{tot}(D)<N(D)$ for the appearing pairs of polynomials. So $c_2(G)=0$. 

If $2h_G\geq N_G$ we iteratively use 
\begin{equation}\label{f10}
\begin{aligned}[]
(a) &\; [ f,g ] \equiv [f^1g_1-f_1g^1] - [f^1,g^1]\mod q\ ,\\
 (b) &\; [fg] \equiv [f] + [g] - [f,g] \mod q\ .
\end{aligned}
\end{equation}

Let $2h_G=N_G$, then the point-counting functions of elements in the form $(f^1,g^1)$ in the linear reduction are zero$\mod q$ for both 3-valent and 4-valent cases, and the resultants $[f,g]_1$ give us the denominator reduction of $G$ in 3-valent case and it's analogue for a 4-valent vertex.
 
 In the case $2h_G>N_G$ we get much more terms then in denominator reduction, for degree $deg_{tot}$ variables we will have in general $4^n$ summands in the reduction. Nevertheless, by the very definition of semilinear reduction, inductively applying steps (\ref{f10}) and decreasing the total degree, we get a final answer.  
 
Assume that on some step we deal with the pair $(f,g)$ with polynomials linear in $\alpha_k$ and the resultant factorizes $[f,g]_k=ab$. Then 
 \begin{equation}
[f,g]=[a]+[b]-[a,b]-[f^k,g^k] \ .
\end{equation}
Since both $a$ and $b$ are of degree much smaller then $deg_{tot}(f,g)-2$, one proves the induction step
\begin{equation}
|[f,g]|\leq 2^{\deg_{tot}(f,g)-1} + 2^{\deg_{tot}(f,g)-2}+|[a]+[b]|\leq 2^{\deg_{tot}(f,g)}
\end{equation}
with the trivial basis of the induction. 
Now, for the 4-valent formula 
\begin{equation}
|c_2(G)|=\leq 2^{2h_G-4} + 4\cdot 2^{2h_G-4}\leq \frac{1}{2} 4^{h_G} \ ,
\end{equation}
and smaller bound for the case of the 3-valent formula.

The only one subtlety here is that the final value $c_2$ is independent of $q$, but only outside of the set of primes of bad reduction. It is possible that in one or several leaves of the tree (T(D)) of the reduction for some element $D\in SLR$ we get $mx$ for some $m\in \ZZ$ with the last variable $x$. This means that the contribution of $[mx]$ to $c_2(G)_q$ is 1 for any $q=p^k$ with $p\not | \;m$ and the contribution is divisible by $p$ when $p|m$, hence $c_2(G)_q$ naturally depends on such situations.
\end{proof}

The most interesting situation is when the degree of appearing polynomials is not too big comparing to the number of variables is the sense $\delta(G)=2h_G-N_G$ is not big. Then it turns out that that $G$ is semilinear reducible for some graphs with small loop number. We ignore the degenerate case when $G$ has a vertex of degree $\leq 2$.

Suppose that $\delta(G)=2$. Then either $G$ has a 3-valent vertex or it is 4-regular graph.  
If $\delta(G)=1$  and $G$ has no vertices of valency more then 4, then it has only 2 3-valent vertices and all the other are 4-valent. Thus the graph is obtained from a 4-regular graph by deletion of 1 of the edges. By formula (\ref{b65}), $c_2$ invariants of these graphs give contribution to $c_2$-invariants  of 4-regular graphs.   

\begin{remark}
There are not very many examples of 4-regular graphs that are semilinear reducible, nevertheless the idea of this reduction can be used for partial reduction: we can get rid of big part of the variables and obtain a polynomial of a much smaller degree depending on less variables, and then apply other techniques or just brute force for that polynomial. 
\end{remark}

The most famous and simple series of graphs in $\phi^4$ theory is $ZZ_h$, and the most simple nontrivial 4-regular graphs are their completions. They should behave nice from the point of view of the $c_2$ invariant. We conjecture the following.
\begin{conjecture}\label{conj25}
Let $ZZ_h$ be the zigzag graph with $h=h_G\geq 3$ loops and let $\widehat{ZZ}_h$ be completion of $ZZ_h$. Then $\widehat{ZZ}_h$ is semilinearly reducible and for any $q$ 
\begin{equation}\label{f33}
c_2(\widehat{ZZ}_n)_q \equiv  - h_G(h_G+2) \mod q 
\end{equation}
(with no primes of bad reduction).
\end{conjecture}

We have checked the conjecture for small loop numbers: 
\begin{proposition}\label{prop26}
The statement of Conjecture \ref{conj25} is true for the 4-regular graphs $\widehat{ZZ_h}$, $h\leq 8$. 
\end{proposition}

There should be an analytic way to prove the semilinear reducibility $\widehat{ZZ}_h$ for all $h$ using the recurrence relations between Dodgson polynomials appearing in the reduction processes.

\end{document}